\documentclass[12pt,a4paper]{article}
\usepackage{amsmath, amsfonts, amsthm,enumerate,DraTex,AlDraTex}
\theoremstyle{plain}
\newtheorem{thm}{Theorem}[section]
\newtheorem{lem}[thm]{Lemma}
\newtheorem{cor}[thm]{Corollary}
\newtheorem{prop}[thm]{Proposition}

\theoremstyle{definition}
\newtheorem{defi}[thm]{Definition}

\title{\bf{Linear groups over a locally linear division ring}}

\author{Bui Xuan Hai\footnote{Faculty of Mathematics and Computer Science, University of Science, VNU-HCM, 227 Nguyen Van Cu Str., Dist. 5, HCM-City, Vietnam,  e-mail: bxhai@hcmus.edu.vn}, Mai Hoang Bien\footnote{Department of Basic Sciences, University of Architecture, 196 Pasteur Str., Dist. 1, HCM-City, Vietnam, e-mail:  maihoangbien012@yahoo.com}, and Trinh Thanh Deo\footnote{Faculty of Mathematics and Computer Science, University of Science, VNU-HCM, 227 Nguyen Van Cu Str., Dist. 5, HCM-City, Vietnam,  e-mail: ttdeo@hcmus.edu.vn}}
\begin{document}
\baselineskip=18pt
\maketitle
\def\Q{\mathbb{Q}}
\def\F{\mathbb{F}}
\newcommand{\dpcm}{ \hfill \rule{3mm}{3mm}}
\newcommand{\hpt}[2]{\left\{\begin{array}{#1} #2\end{array}\right.}
\begin{abstract}
In this paper, in the first we give definitions of some classes of division rings which strictly contain the class of centrally finite division rings. One of our main purpose is to construct  non-trivial examples of rings of new defined classes. Further, we study linear groups over division rings of these classes. Our  new obtained results generalize precedent results for centrally finite division rings.
\end{abstract}

{\bf {\em Key words:}}  Division ring; algebraic; strongly algebraic;  locally linear; linear groups. 

{\bf{\em  Mathematics Subject Classification 2010}}: 16K20, 16K40 

\newpage
\section{Introduction}

Let $D$ be a division ring and $F$ be its center. Recall that $D$ is {\em centrally finite} if $D$ is a finite dimensional vector space over $F$; $D$ is {\em locally centrally finite} if for every finite subset $S$ of $D$, the division subring $F(S)$  of $D$ generated by $S$ over $F$ is a finite dimensional vector space over $F$. If $a$ is an element from $D$, then we have the field extension $F\subseteq F(a)$. Obviously, $a$ is {\em algebraic} over $F$ if and only if this extension is finite. We say that a non-empty subset $S$ of $D$ is {\em algebraic} over $F$ if every element of $S$ is algebraic over $F$. A division ring $D$ is  {\em algebraic} over the center $F$ (briefly, $D$ is {\em algebraic}), if every element of $D$ is algebraic over $F$. Clearly, the class of algebraic division rings contains the class of locally centrally finite division rings and the last class contains the class of centrally finite division rings. It is not difficult to give examples showing that these classes are different. In this paper we give the definition of the class  of so called {\em strongly algebraic division rings}, which lies between the class of centrally finite division rings and the class of algebraic division rings. Also, we define the classes of {\em linear} and {\em locally linear} division rings. The relation between these classes is explained in Section 2. One of our main purposes is to construct in Section 2 the non-trivial examples of rings belonging to our new defined classes. Section 3 is  devoted to the study of subgroups in locally linear division rings. In Section 4 we shall investigate some properties of linear groups over division rings of these new defined classes. Our new obtained results generalize precedent results for centrally finite division rings. The symbols and notation we use in this paper are standard and they should be found in the literature on subgroups in division rings and on skew linear groups.

\section{Definitions and examples}

\begin{defi}\label{def:1.1} 
Let $D$ be a division ring. 
\begin{enumerate}[i)]
  \item We say that $D$ is a {\it linear division ring } if $D$ can be embedded in some centrally finite division ring.  
  \item  We say that $D$ is a {\it locally linear division ring } if  for every finite subset $S$ of $D$, the division subring of $D$ generated by $S$ is linear.
\end{enumerate}
\end{defi}

\begin{defi}\label{def:1.2}
Let $D$ be a division ring which is algebraic over its center $F$. We say that $D$ is {\it strongly algebraic} over $F$ if $D$ contains a maximal subfield $K$ satisfying the following conditions:
\begin{enumerate} [i)]
  \item there exists a subset $S$ of $K$ such that $K=F(S)$,
  \item for each $x\in D$, there are at most finitely many elements $y$ in $S$ such that $xy\neq yx$. 
\end{enumerate}  
\end{defi}

\begin{prop}\label{prop:1.3}
Every centrally finite division ring is strongly algebraic.
\end{prop}
\begin{proof}
Suppose that $D$ is a centrally finite division ring with center $F$. By [\cite{dra}, \S7, Th. 4, p.45], there exists a maximal subfield  $K$ of $D$ containing $F$.  
Since $[K:F]<\infty$, there exists a finite subset $S$ such that $K=F(S)$. Now, by definition we see that $D$ is strongly algebraic over $F$. 
\end{proof}

\begin{thm} \label{thm:1.4}
Let $D$ be a strongly algebraic division ring over its center $F$ and $T$ be a finite subset of $D$. Then, there exists some division subring of $D$ containing $F$ which is  centrally finite by itself and contains $T$. 
\end{thm}
\begin{proof}
Suppose that $K=F(S)$ is a maximal subfield of $D$ such that for every element $x$ in $D$, there are at most finitely many elements $y$ in $S$ such that $xy\neq yx$. 
Denote by $U$ the set of elements $s$ in $S$ such that $s$ does not commute at least with one element of $T$. Then, $U$ is a finite subset of $S$ and every element of   $S\setminus U$ commutes with all elements from $T$. Therefore,  $F(S\setminus U)\subseteq Z(K(T))$. Note that, since $K$ is a maximal subfield of $D, K$ is also a maximal subfield of  $K(T)$. It follows that $Z(K(T))\subseteq K$. Further, 
since $K=F(S)=F((S\setminus U)\cup U)=F(S\setminus U)(U)$ and  $U$ is a finite subset algebraic over $F$, we have $[K:F(S\setminus U)]<\infty$. Consequently, $[K:Z(K(T))]<\infty$. Now, by [\cite{lam},(15.8), p.255],  $K(T)$ is centrally finite and obviously, this fact completes  the proof of the theorem.
\end{proof}  

\begin{cor}\label{cor:1.5}
If a division ring $D$ is  strongly algebraic over its center, then $D$ is locally linear. 
\end{cor}

Our next purpose in this section is to construct some examples showing the difference between  new defined classes and the precedent classes of division rings. In order to do so, first, we shall construct division subrings of the ring $D=K((G, \Phi))$, which was introduced in   \cite{dbh}. 

Namely, if we denote by $G=\bigoplus\limits_{i=1}^\infty \mathbb{Z}$ the direct sum of infinitely many of copies of the additive group $(\mathbb{Z}, +)$ of all integers , then  $G$ is the  set of all infinite sequences of integers of the form  $(n_1, n_2, n_3, \ldots)$ with only finitely many non-zeros  $n_i$. For any positive integer $i$, denote by $x_i= (0, \ldots, 0, 1, 0, \ldots)$ the element of $G$ with  $1$ in the $i$-th position and $0$ elsewhere.  Then $G$ is a free abelian group generated by all $x_i$ and every element $x$ in $G$ is written uniquely in the form
$$x=\sum\limits_{i\in I}n_i x_i,$$
with $n_i\in\mathbb{Z}$ and some finite set $I$. 

Now, we define an order in $G$ as follows:

For elements  $x=(n_1, n_2, n_3, \ldots)$ and  $y=(m_1, m_2, m_3, \ldots)$ in $G$, define  $x<y$ if either  $n_1<m_1$  or there exists $k\in \mathbb{N}$ such that  $n_1=m_1,\ldots, n_k=m_k$ and  $n_{k+1}<m_{k+1}$. Clearly, with this order $G$ is a totally ordered set.

Suppose that  $p_1<p_2<\ldots <p_n<\ldots$ is a sequence of prime numbers and 
$K=\mathbb{Q}(\sqrt{p_1},\sqrt{p_2},\ldots)$ is the subfield of the field $\mathbb{R}$ of real numbers generated by $\mathbb{Q}$ and  $\sqrt{p_1},\sqrt{p_2},\ldots$, where $\mathbb{Q}$ is the field of rational numbers. For any $i\in\mathbb{N}$, suppose that  $f_i:K\to K$ is  $\mathbb{Q}$-isomorphism satisfying the following condition: 
$$f_i(\sqrt{p_j})=\hpt{rl}{\sqrt{p_j},&\text{if }j\neq i;\\ -\sqrt{p_i},&\text{if }j=i.}$$ 

It is easy to verify that $f_i f_j=f_j f_i,\forall i,j\in\mathbb{N}.$ Moreover, we have the following lemma, whose proof can be found in \cite{dbh}:

\begin{lem}\label{lem:2.1}
Suppose that $x\in K$. Then, $f_i(x)=x$ for any $i\in \mathbb{N}$ if and only if  $x\in\mathbb{Q}$.
\end{lem}

For an element  $x=(n_1, n_2, ...)=\sum\limits_{i\in I} n_i x_i\in G$, define $\Phi_x:=\prod\limits_{i\in I} f_i^{n_i}.$ Clearly $\Phi_x\in Gal(K/\mathbb{Q})$ and the map
$\Phi: G\to Gal(K/\mathbb{Q}),$
defined by $\Phi(x)=\Phi_x$ is a group homomorphism. It is easy to prove the following proposition:

\begin{prop}\label{prop:2.2}
\begin{enumerate}[{\rm i)}]
  \item $\Phi(x_i)=f_i$ for any $i\in \mathbb{N}$. 
  \item If $x=(n_1, n_2, \ldots)\in G$, then $\Phi_x(\sqrt{p_i})=(-1)^{n_i} \sqrt{p_i}$.
\end{enumerate}
\end{prop}

For the convenience, from now on we write the operation in  $G$ multiplicatively. For $G$ and $K$ as above, consider formal sums of the form
$$\alpha=\sum\limits_{x\in G} a_x x,\quad a_x\in K.$$

For  such an $\alpha$, define the support of $\alpha$ by $supp(\alpha)=\{x\in G: a_x\neq 0\}$. Put
 $$D=K((G,\Phi)):=\Big\{\alpha=\sum\limits_{x\in G} a_x x, a_x\in K ~\vert~ supp(\alpha) \text{ is  well-ordered }\Big\}.$$ 

For  $\alpha=\sum\limits_{x\in G} a_x x$ and $\beta=\sum\limits_{x\in G} b_x x$ from $D$, define 
\begin{eqnarray*}
\alpha+\beta&=&\sum\limits_{x\in G} (a_x+b_x) x;\\
\alpha.\beta&=&\sum\limits_{z\in G} \Big(\sum\limits_{xy=z}a_x \Phi_x(b_y)\Big) z.
 \end{eqnarray*}

In [\cite{lam}, p.243], it is proved that these operations are well-defined. Moreover, the following theorem holds: 
\begin{thm}[\cite{lam}, Th.(14.21), p.244] \label{thm:2.3}
$D=K((G,\Phi))$ with the operations, defined as above is a division ring.
\end{thm}

Put $H:=\{x^2: x\in G\}$ and
$$\mathbb{Q}((H)):=\Big\{\alpha=\sum\limits_{x\in G} a_x x, a_x\in \mathbb{Q} ~\vert~ supp(\alpha) \text{ is  well-ordered }\Big\}.$$ 

It is easy to check that  $H$ is a subgroup of $G$ and for every $x\in H$, $\Phi_x=Id_K$. Note that in  [\cite{dbh}, Theorem 2.2] it was proved that  $\mathbb{Q}((H))$ is the center of $D$.  

Now, for $n\geq 1$, denote by $L_n:=F(\sqrt{p_1},\ldots, \sqrt{p_n}, x_1, \ldots, x_n)$. Then, $L_n\subseteq L_{n+1}$ and $L_{\infty}:=\bigcup\limits_{n=1}^{\infty}L_n$ is the division subring generated by all $\sqrt{p_i}$ and all $x_i$ over $F$.

\begin{prop}\label{prop:2.7}
The division ring  $L_\infty$ satisfies the following conditions:
\begin{enumerate}[{\rm i)}]
  \item $L_\infty$ is locally centrally finite.
  \item $L_\infty$ is strongly algebraic over its center.
  \item $L_\infty$ is not linear.
\end{enumerate}
\end{prop}
\begin{proof}
i) For a finite subset $S\subseteq L_\infty$, since  $L_{\infty}=\bigcup\limits_{n=1}^{\infty}L_n$, there exists some  $n$ such that $S\subseteq L_n$. By Lemma 3.1 and Theorem 3.3 in \cite{dbh}, $Z(L_n)=Z(L_{\infty})=F$ and $L_n$ is centrally finite. Consequently, the division subring of $L_{\infty}$ generated by $S$ over $F$ is finite dimensional over $F$. Hence $L_\infty$ is locally centrally finite.

ii) By i), $L_\infty$ is locally  centrally finite with the center $F$. So, $L_\infty$ is algebraic over $F$. Denote by $K_\infty =F(\sqrt{p_1},\sqrt{p_2},\ldots)$ the subfield of $L_\infty$ generated by $\sqrt{p_1},\sqrt{p_2},\ldots$ over $F$ and 
suppose that   $\alpha\in C_{L_\infty}(K_\infty)\setminus K_\infty$. Then, there exists some $i$ such that $x_i$ appears in the expression of $\alpha$ as a formal sum. Since $x_i^2\in F$,   $\alpha$ can be expressed in the form
$\alpha= \beta x_i+\gamma$, where  $\beta\neq 0$ and  $x_i$ does not appear in the formal expressions of $\beta$ and  $\gamma$.
Therefore, $\sqrt{p_i}\alpha-\alpha \sqrt{p_i}=2\beta\sqrt{p_i} x_i\neq 0$. It follows that $\alpha$ does not commute with  $\sqrt{p_i}\in K_{\infty}$, that is a contradiction. Hence, $C_{L_\infty}(K_\infty)=K_\infty$. 
In  view of [\cite{lam}, Prop. (15.7),p.254], we have
$K_{\infty}$ is a maximal subfield of $L_\infty$. 

Moreover, we have $K_\infty = F(S)$, where $S=\{\sqrt{p_1}, ..., \sqrt{p_n}, ...\}$.
Since  $L_\infty=\bigcup\limits_{n=1}^\infty L_n$, for  $\alpha\in L_\infty$, there exists  $n$ such that  $\alpha\in L_n$. 
From the proof of [\cite{dbh}, Lemma 3.1 i)], we see that every element of  $L_n$ can be expressed in the form 
$$\alpha = \sum_{0\leq \varepsilon_i, \mu_i\leq 1} a_{(\varepsilon_1, ..., \varepsilon_n,  \mu_1, ..., \mu_n)} (\sqrt{p_1})^{\varepsilon_1}\ldots (\sqrt{p_n})^{\varepsilon_n} x_1^{\mu_1} \ldots x_n^{\mu_n},\quad  a_{(\varepsilon_1, ..., \varepsilon_n,  \mu_1, ..., \mu_n)}\in F.$$
Therefore, $\alpha$ commutes with each from elements $\sqrt{p_{n+1}}$, $\sqrt{p_{n+2}},\ldots$ .  By definition we see that 
$L_\infty$ is strongly algebraic over $F$.

iii) Suppose that $L_\infty$ is linear. This means that, there exists some centrally finite division ring $L$ such that $L_\infty\subseteq L$. Since $x_i$ does not commute with  $\sqrt{p_i}, x_i\not\in Z(L)$. We claim that the set 
$B=\{x_i: i=1,2,...,\}$ is linearly independent over $Z(L)$. Thus, suppose that 
$B$ is linearly dependent over  $Z(L)$. Then, there exists some $n$ such that $x_n$ is a linear combination  of  $x_1, \ldots, x_{n-1}$ over $Z(L)$. Since  $\sqrt{p_n}$ commutes with $x_1, \ldots, x_{n-1},$ $\sqrt{p_n}$ commutes with $x_n$, which is a contradiction. Thus,  $B$ is an infinite linearly independent set over  $Z(L)$ that is impossible. 
\end{proof} 

Suppose that  $\alpha = x_1^{-1}+x_2^{-1}+\ldots$ is an infinite sum.  
Since  $x_1^{-1}<x_2^{-1}<\ldots,  supp(\alpha)$ is  well-ordered. Hence $\alpha\in D$. Put 
$$R_n=L_n(\alpha)=F(\sqrt{p_1},\sqrt{p_2},\cdots, \sqrt{p_n},x_1,x_2,\cdots x_n, \alpha),\forall n\ge 1;$$
and  $$R_\infty=\bigcup_{n=1}^\infty R_n.$$

The following theorem holds:
\begin{prop}\label{md-2.3}
The division ring $R_\infty$ is locally linear and it is not algebraic over its center.  
\end{prop}
\begin{proof} 
First, we prove that $R_n$ is centrally finite for each positive integer $n$.
Consider the element  $$\alpha_n=x_{n+1}^{-1}+x_{n+2}^{-1}+\cdots \quad \text{(infinite sum)}.$$ 
Since $\alpha_n =\alpha -(x_1^{-1}+x_2^{-1}+\cdots+x_n^{-1}),\alpha_n\in R_n$ and 
$$F(\sqrt{p_1},\sqrt{p_2},\ldots, \sqrt{p_n},x_1,x_2,\ldots x_n, \alpha)=F(\sqrt{p_1},\sqrt{p_2},\ldots, \sqrt{p_n},x_1,x_2,\ldots x_n, \alpha_n).$$
Note that $\alpha_n$ commutes with all  $\sqrt{p_i}$ and all $x_i$ (for $i=1,2,...,n$).  Therefore
\begin{eqnarray*}
R_n&=&F(\sqrt{p_1},\sqrt{p_2},\cdots, \sqrt{p_n},x_1,x_2,\cdots x_n, \alpha_n)\\
&=&F(\alpha_n)(\sqrt{p_1},\sqrt{p_2},\cdots, \sqrt{p_n},x_1,x_2,\cdots x_n).
\end{eqnarray*}
In combination with the equalities  
$$(\sqrt{p_i})^2=p_i, x_i^2\in F, \sqrt{p_i}x_j=x_j\sqrt{p_i}, i\ne j, \sqrt{p_i}x_i=-x_i\sqrt{p_i},$$
it follows that every element $\beta$ of  $R_n$ can be written in the form 
$$\beta = \sum\limits_{0\leq \varepsilon_i, \mu_i\leq 1} a_{(\varepsilon_1, ..., \varepsilon_n,  \mu_1, ..., \mu_n)} (\sqrt{p_1})^{\varepsilon_1}\ldots (\sqrt{p_n})^{\varepsilon_n} x_1^{\mu_1} \ldots x_n^{\mu_n},$$
where 
$$a_{(\varepsilon_1, ..., \varepsilon_n\mu_1, ..., \mu_n)}\in F(\alpha_n).$$
Hence  $R_n$ is a vector space over  $F(\alpha_n)$ having the finite set $B_n$ which consists of the products  
$$(\sqrt{p_1})^{\varepsilon_1}\ldots (\sqrt{p_n})^{\varepsilon_n} x_1^{\mu_1} \ldots x_n^{\mu_n}, 0\le \varepsilon_i, \mu_i\le 1$$
 as a base.
Thus,  $R_n$ is a finite dimensional vector space over $F(\alpha_n)$. 
Since  $\alpha_n$ commutes with all  $\sqrt{p_i}$ and all $x_i, F(\alpha_n)\subseteq Z(R_n)$. It follows that  $\dim_{Z(R_n)}R_n\le \dim_{F(\alpha_n)}R_n<\infty$ and consequently, $R_n $ is centrally finite. 

For any finite subset $S\subseteq R_\infty$, there exists  $n$ such that $S\subseteq R_n$. Therefore, the division subring of $R_\infty$, generated by $S$ over $F$ is contained in $R_n$, which is centrally finite. Thus,  $R_\infty$ is locally linear. 

On the other hand, by [\cite{dbh}, Theorem 3.3], we have  $Z(L_\infty)=F$. So, $Z(R_\infty)=Z(L_\infty(\alpha))=F$. The proof of [\cite{dbh}, Theorem 3.1] shows that  $\alpha $ is not algebraic over $F$, hence  $R_\infty$ is not algebraic over $F$. 
\end{proof}

According to the theorem above, the division ring  $R_\infty$ is locally linear. However, it is not algebraic  and consequently it is not strongly algebraic over its center.  

Finally, we strongly believe that  there exists a division ring which is algebraic but it is not strongly algebraic over its center. However, at the present time, we do not have any counterexample to this problem. On the other hand, we believe that any division subring of a centrally finite division ring is itself centrally finite. In fact, we propose the following conjecture:

\noindent
{\bf Conjecture.} {\em Any division subring of a centrally finite division ring is itself centrally finite.}

\section{Locally linear division rings}

\begin{thm}\label{thm:3.2}
If  $D$ is a locally linear division ring, then $Z(D')$ is a torsion group.
\end{thm}
\begin{proof} By [\cite{hdb}, Proposition 2.1], $Z(D')=D'\cap F$. For any $x\in Z(D')$, there exists some positive integer $n$ and some $a_i, b_i\in D^*, 1\leq i\leq n$, such that 
$$x=a_1b_1a_1^{-1}b_1^{-1}a_2b_2a_2^{-1}b_2^{-1}\cdots a_nb_na_n^{-1}b_n^{-1}.$$

Set $S:=\{a_i,b_i: 1\leq i\leq n\}$.
Since $D$ is locally linear, there exists a centrally finite division ring $L$  such that  $P(S)$ is embedded in $L$. Clearly, we can suppose that  $S\subseteq L.$
Put $K=Z(L), L_1=K(S)$ and  $F_1=Z(L_1)$. 
Since $K\subseteq L_1\subseteq L$ and  $dim_K L <\infty$, it follows that $K\subseteq F_1$ and  $dim_K L_1<\infty$. Hence $n=\dim_{F_1}L_1<\infty.$ 
On the other hand, since $x\in F, x$ commutes with every element of $S$. Therefore,  $x$ commutes with every element of  $L_1=K(S)$, and consequently,  $x\in F_1$. So,  
$$x^n=N_{L_1/F_1}(x)=N_{L_1/F_1}(a_1b_1a_1^{-1}b_1^{-1}a_2b_2a_2^{-1}b_2^{-1}\cdots a_nb_na_n^{-1}b_n^{-1})=1.$$ 
Thus, $x$ is torsion.
\end{proof}

The following corollary carries over one of Herstein's result [\cite{her}, Theorem 2] to the case of locally linear division rings.  
\begin{cor}\label{cor:3.3}
Let $D$ be a locally linear division ring with the center $F$. If for any   $a,b\in D^*$, there exists a positive integer  $n=n_{ab}$ depending on $a$ and $b$, such that  $(aba^{-1}b^{-1})^n\in F$, then  $D$ is commutative.
\end{cor}
\begin{proof}
By Theorem \ref{thm:3.2},  for any  $a,b\in D^*, aba^{-1}b^{-1}$ is torsion . By [\cite{her}, Theorem 1], $D$ is commutative.
\end{proof}
Using Theorem  \ref{thm:3.2}, it is easy to prove that Conjecture 3 in \cite{her} is true for locally linear division rings. In fact, we have the following result:
\begin{thm}\label{thm:3.4}
Let $D$ be a locally linear division ring with the center $F$ and  $N$ be a  subnormal subgroup of  $D^*$. If  $N$ is radical over $F$, then $N$ is central, i.e. $N$ is contained in $F$.
\end{thm}
\begin{proof}
Consider the subgroup  $N'=[N,N]\subseteq D'$ and suppose that $x\in N'$. Since $N$ is radical over $F$, there exists some positive integer  $n$ such that $x^n\in F$. Hence $x^n \in F\cap D'=Z(D')$. By Theorem  \ref{thm:3.2},  $x^n$ is torsion, and consequently, $x$ is torsion too. Moreover, since  $N$ is subnormal in $D^*$, so is $N'$. Hence, by  [\cite{her}, Th. 8], $N'\subseteq F$. Thus, $N$ is solvable, and by [\cite{scott}, 14.4.4, p. 440], $N\subseteq F$.
\end{proof}

Now, we study subgroups of $D^*$, that are radical over some subring of $D$. To prove the next theorem we need the following useful property of locally linear division rings.

\begin{lem}\label{lem:3.5} 
Let $D$ be a locally linear division ring with the center $F$ and $N$ be a subnormal subgroup of $D^*$. If for every elements $x, y\in N$, there exists some positive integer $n_{xy}$ such that $x^{n_{xy}}y=yx^{n_{xy}}$, then $N\subseteq F$.
\end{lem}
\begin{proof} 
Replacing  $K:=F(x, y)$ by $K:=P(x, y)$ ($P$ is the prime subfield of $F$) in  the proof of [\cite{hdb}, Lem. 3.1], we can obtain similar proof of this lemma for the case of locally linear division ring instead of the case of division rings of type $2$. 
\end{proof}

Using this lemma, by the same way as in \cite{hdb}, we can prove the following non-trivial theorem whose proof is identified with the proof of Theorem 3.1 in \cite{hdb} for division rings of type $2$. 
\begin{thm}\label{thm:3.6} 
Let $D$ be a locally linear division ring with the center $F,K$ be a proper division subring of $D$ and suppose that $N$ is a normal subgroup of $D^*$. If $N$ is radical over $K$, then $N\subseteq F$.
\end{thm}

\section{Finitely generated skew linear groups}

The following result generalizes Theorem 1 in  \cite{mah1}. 
\begin{thm} \label{thm:4.1}
Let $D$ be a division ring strongly algebraic over its center $F$ and $N$ be a subnormal subgroup of $D^*$.  If $N$ is finitely generated, then $N$ is central.
\end{thm}
\begin{proof}
Since $N$ is finitely generated, by Theorem \ref{thm:1.4}, there exists some centrally finite division subring $L$ of $D$ such that  $N\subseteq L$.
By [\cite{mah1}, Th.1], $N\subseteq Z(L)$. Consequently, $N$ is abelian. Now, by 
[\cite{scott}, 14.4.4, p. 440], $N\subseteq F$.
\end{proof}  

In the following we identify $F^*$ with $F^*I:=\{\alpha I\vert~ \alpha\in F^*\}$, where $I$ denotes the identity matrix in $GL_n(D)$. 
\begin{thm}\label{thm:4.2}
Let $D$ be a division ring strongly algebraic over its center $F$ and $N$ be a infinite  subnormal subgroup of $GL_n(D)$ with  $n\geq 1$. If $N$ is finitely generated, then  $N\subseteq F$.
\end{thm}
\begin{proof}
If $n=1$, then the result follows from Theorem \ref{thm:4.1}. 

Suppose that $n>1$ and $N$ is non-central. Then, by  [\cite{mah2}, Th.4], $SL_n(D)\subseteq N$. So, $N$ is normal in  $GL_n(D)$. Suppose that $N$ is generated by matrices $A_1, A_2, ..., A_k$ in $GL_n(D)$ and $T$ is the set of all coefficients of all $A_j$. By Theorem  \ref{thm:1.4}, there exists some centrally finite division subring $L$ of $D$ containing $T$. It follows that $N$ is a normal finitely generated subgroup of $GL_n(L)$. By [\cite{akb3}, Th.5], $N\subseteq Z(GL_n(L))$.
In particular,  $N$ is abelian and consequently, $SL_n(D)$ is abelian, that is a contradiction. 
\end{proof}

\begin{lem}\label{lem:4.6}
Let $D$ be a division ring with center $F$ and $n\geq 1$. Then, $Z(SL_n(D))$ is a torsion group if and only if $Z(D')$ is  a torsion group.
\end{lem}
\begin{proof}
The case $n=1$ is clear. So, we can assume that $n\geq 2$. By [\cite{dra}, \S21, Th.1, p.140], 
$$Z(SL_n(D))=\big\{ dI \vert d\in F^* \text{ and } d^n\in D'\big\}.$$ 
If $Z(SL_n(D))$ is a torsion group, then, for any $d\in Z(D')=D'\cap F$, $dI\in Z(SL_n(D))$. It follows that $d$ is periodic.
Conversely, if $Z(D')$ is a torsion group, then, for any $A\in Z(SL_n(D))$, $A =dI$ for some $d\in F^*$ such that $d^n\in D'$. It follows that $d^n$ is periodic. Therefore, $A$ is periodic.
\end{proof}

Now we can prove the following theorem, which shows that if $D$ is a non-commutative division ring which strongly algebraic over its center, then  there are no finitely generated subgroups of $GL_n(D)$, containing $F^*$. 

\begin{thm}\label{thm:4.7}
 Let $D$ be a non-commutative division ring which is strongly algebraic over its center $F$ and  $N$ be a subgroup of $GL_n(D)$ containing $F^*$, $n\geq 1$.  Then $N$ is not finitely generated.
\end{thm}
\begin{proof} Recall  that if $D$ is strongly algebraic over its center, then $Z(D')$ is a torsion group (see Corollary \ref{cor:1.5} and Theorem \ref{thm:3.2}). Therefore, by Lemma \ref{lem:4.6}, $Z(SL_n(D))$ is a torsion group.

Suppose that there is  a finitely generated subgroup $N$ of $GL_n(D)$ containing $F^*$. Then, in virtue of [\cite{scott}, 5.5.8, p. 113], $F^*N'/N'$ is a finitely generated abelian group, where $N'$ denotes the derived subgroup of $N$.\\

\noindent
{\em Case 1: $char(D)=0$.}

 Then, $F$ contains the field $\Q$ of rational numbers and it follows that $\Q^*I/(\Q^*I\cap N')\simeq \Q^*N'/N'$. Since   $F^*N'/N'$ is finitely generated, $\Q^*N'/N'$ is finitely generated and consequently   $\Q^*I/(\Q^*I\cap N')$ is finitely generated. Considering an arbitrary $A\in \Q^*I\cap N'$. Then $A\in F^*I\cap SL_n(D)\subseteq Z(SL_n(D))$. 
Therefore $A$ is periodic.
Since $A\in \Q^*I$, we have $A=dI$ for some $d\in \Q^*$. It follows that $d=\pm{1}$. Thus, $\Q^*I\cap N'$ is finite. Since $\Q^*I/(\Q^*I\cap N')$ is finitely generated, $\Q^*I$ is finitely generated. Therefore $\Q^*$ is finitely generated, that  is impossible.\\

\noindent
{\em Case 2: $char(D)=p > 0$.}

 Denote by $\F_p$ the prime subfield of $F$, we shall prove that $F$ is algebraic over $\F_p$. In fact, suppose that $u\in F$ and $u$  is transcendental over $\F_p$. Put $K:=\F_p(u)$, then the group $K^*I/(K^*I\cap N')$ considered as a subgroup of $F^*N'/N'$ is finitely generated. Considering an arbitrary $A\in K^*I\cap N'$, we have $A=(f(u)/g(u))I$ for some $f(X), g(X)\in \F_p[X], ((f(X), g(X))=1$ and $g(u)\neq 0$. As mentioned above, we have $f(u)^s/g(u)^s=1$ for some positive integer $s$. Since $u$ is transcendental over $\F_p$,  $f(u)/g(u)\in \F_p$. Therefore, $K^*I\cap N'$ is finite and consequently, $K^*I$ is finitely generated. It follows that $K^*$ is finitely generated.
This contradicts to  [\cite{hdb}, Lem. 2.2], stating that $K$ is finite. Hence $F$ is algebraic over $\F_p$ and it follows that $D$ is algebraic over $\F_p$. Now, in virtue of Jacobson's Theorem  [8, (13.11), p. 219],  $D$ is commutative, which completes the proof. 
\end{proof}
 
From   Theorem \ref{thm:4.7} we  get the following result, which  generalizes  Theorem 1 in \cite{akb3}:
\begin{cor}\label{cor:4.8}
Let $D$ be a division ring which strongly algebraic over its center. If the group $GL_n(D)$ is finitely generated, then $D$ is commutative.
\end{cor}

If $M$ is a maximal  finitely generated subgroup of $GL_n(D)$, then $GL_n(D)$ is finitely generated. So, the next result follows  immediately from  Corollary \ref{cor:4.8}.

\begin{cor}\label{cor:4.9}
Let $D$ be a division ring which strongly algebraic over its center. If the group $GL_n(D)$  has a maximal  finitely generated subgroup, then $D$ is commutative.
\end{cor}

By the same way as in the proof of Theorem \ref{thm:4.7}, we obtain the following corollary.

\begin{cor}\label{cor:4.10}
Let $D$ be a non-commutative division ring which strongly algebraic over its center $F$ and $S$ is a  subgroup of $GL_n(D)$. If $N=F^*S$, then $N/N'$ is not finitely generated.
\end{cor}

\begin{proof}
Suppose that $N/N'$ is finitely generated. Since $N'=S'$ and $ F^*I/(F^*I\cap S') \simeq F^*S'/S'$, $F^*I/(F^*I\cap S')$ is a finitely generated abelian group.  Now, by the same arguments as in the proof of Theorem \ref{thm:4.7}, we conclude that $D$ is commutative.
\end{proof}

The following result follows immediately from Corollary \ref{cor:4.10}.

\begin{cor}\label{cor:4.11}
Let $D$ be a non-commutative division ring which strongly algebraic over its center. Then, $GL_n(D)/SL_n(D)$  is not finitely generated.
\end{cor}

\end{document}